\documentclass[11pt]{amsart}

\usepackage[utf8]{inputenc}
\usepackage[T1]{fontenc,url}

\usepackage[left=3.3cm,right=3.3cm,bottom=2.3cm,top=2.3cm]{geometry}

\usepackage{amsmath,amsthm,amsfonts,amssymb}
\usepackage{mathrsfs}
\usepackage[shortlabels]{enumitem}
\usepackage{cite}
\usepackage{mathtools}
\usepackage{color}
\definecolor{shadecolor}{gray}{0.875}
\definecolor{col}{RGB}{42, 95, 151}
\usepackage[colorlinks=true, linkcolor=col, citecolor=col, filecolor=col, menucolor=col, urlcolor=col]{hyperref}

\numberwithin{equation}{section}

\theoremstyle{plain}
\newtheorem{theorem}{Theorem}[section]
\newtheorem*{lemma*}{Lemma}
\newtheorem{lemma}[theorem]{Lemma}
\newtheorem*{theorem*}{Theorem}
\newtheorem{proposition}[theorem]{Proposition}
\newtheorem*{proposition*}{Proposition}
\newtheorem{corollary}[theorem]{Corollary}
\newtheorem*{corollary*}{Corollary}

\theoremstyle{definition}
\newtheorem{remark}[theorem]{Remark}
\newtheorem*{remark*}{Remark}
\newtheorem*{definition*}{Definition}

\newtheorem*{example*}{Example}

\newtheorem{example}[theorem]{Example}

\newtheorem*{question*}{Question}

\def\rank{\operatorname{rank}}

\def\c1{\operatorname{c_1}}
\def\c2{\operatorname{c_2}}
\def\Spec{\operatorname{Spec}}

\def\CC{{\mathbb C}}
\def\ZZ{{\mathbb Z}}

\def\PP{{\mathbb P}}

\def\O{{\mathcal O}}

\def\AA{{\mathbb A}}

\def\+{\oplus}
\def\*{\otimes}

\def\Bl{\operatorname{Bl}}

\DeclareMathOperator{\Br}{Br}

\newcommand{\CH}{\operatorname{CH}}
\newcommand{\nr}{\mathrm{nr}}
\newcommand{\dd}{\partial}

\title[Quartic sixfolds]{Very general quartic sixfolds\\ are not stably rational}

\author{John Christian Ottem}
 \address{Department of Mathematics,
 University of Oslo,
 Box 1053, Blindern,
 0316 Oslo, Norway}
 \email{johnco@math.uio.no}

\begin{document}

\begin{abstract}
We prove that a very general quartic sixfold admits no decomposition of the diagonal, and in particular is not stably rational.
The proof uses specialization to an explicit quartic sixfold birational to a quadric bundle of the type considered by Colliot-Th\'el\`ene and Ojanguren.
We also prove stable irrationality for very general six-dimensional intersections of four quadrics.
\end{abstract}
\maketitle



\section{Introduction}
A complex variety $X$ is called {\em rational} if it is birational to projective space and {\em stably rational} if $X\times \PP_\CC^m$ is rational for some $m\geq0$.
The main result of this paper is that a very general hypersurface of degree 4 and dimension 6 is not stably rational. 
\begin{theorem}\label{thm:main}
A very general quartic hypersurface $X\subset\PP^7_\CC$ admits no decomposition of the diagonal. In particular, $X$ is not stably rational.
\end{theorem}
Here {\em very general} means outside a countable union of proper closed subsets of the parameter space. 
A projective variety $X$ of dimension $n$ admits a {\em decomposition of the diagonal} if
$$
 [\Delta_X]=[X\times z]+[Z],
$$in the Chow group $\CH_n(X\times X)$, where $z$ is a zero-cycle of degree 1 and $Z$ does not dominate
the first factor. Every smooth projective stably rational variety admits such a
decomposition.

The decomposition of the diagonal property is useful in rationality problems because it specializes in families. That is, if a smooth projective variety $X$ specializes to a variety
$X_0$ and $X_0$ admits no such decomposition, then neither does $X$, so $X$ cannot be stably rational.
This method was introduced by Voisin \cite{Voi} and developed by Colliot-Th\'el\`ene--Pirutka \cite{CTP} and Schreieder \cite{Sch}.

The rationality problem for quartic hypersurfaces has a long history. Iskovskikh--Manin \cite{IM} proved that every smooth quartic threefold is irrational. Using the above specialization method, Colliot-Th\'el\`ene--Pirutka \cite{CTP} 
proved that a  very general quartic threefold admits no decomposition of the diagonal. 
Totaro \cite{Tot} proved the analogous statement for very general quartic fourfolds. 
Nicaise and the author \cite{NO} proved that very general quartic fivefolds are not stably rational, using toric degenerations. This was extended by Pavi\'c--Schreieder \cite{PS} to the nonexistence of a decomposition of the diagonal.
See Schreieder \cite{SchHyp} and Lange--Schreieder \cite{LS} for irrationality bounds for hypersurfaces in higher dimensions. Theorem~\ref{thm:main} settles the first open case for quartics.

As in \cite{LS,PS,Sch,SchHyp}, we use unramified cohomology to obstruct having a decomposition of the diagonal. 
Unramified cohomology first appeared in rationality problems in the work of Artin--Mumford \cite{AM} on the Brauer group.
Here we use degree 3 unramified cohomology, closely following the work of Colliot-Th\'el\`ene--Ojanguren \cite{CTO}.


Let $K=\CC(x,y,z)$ and let $g_1,g_2\in K^\times$. Colliot-Th\'el\`ene--Ojanguren considered quadrics of the form
\begin{equation}\label{eq:Q}
Q=\left\{S^2-xT^2-yU^2+xyV^2-g_1g_2W^2=0\right\}
\subset\PP^4_K
\end{equation}
and gave a criterion for constructing a nonzero class in $H^3_\nr(K(Q),\ZZ/2)$.

  

The starting point of the present paper was to search for $g_1,g_2$ such that the equation above becomes a quartic after a change of variables. This happens for the cubics
\begin{equation}\label{eq:sparse}
g_1=yz^2+xy+2xz+1, \qquad  g_2=yz^2+xy-2yz+1.\end{equation}
Here the difference $g_1-g_2=2z(x+y)$ has degree 2. On the $W=1$ chart, letting $s=w+(g_1+g_2)/2$ therefore makes the higher degree terms cancel and gives the affine quartic
\begin{equation}\label{eq:introaffine}
w^2+2(yz^2+xy+xz-yz+1)w
-xt^2-yu^2+xyv^2+z^2(x+y)^2=0.
\end{equation}
The projective closure of this is a quartic sixfold $X_0\subset\PP^7_\CC$ singular along a 3-plane.
Blowing up this 3-plane we obtain a sixfold $Y_0$ with a morphism $Y_0\to \PP^3_\CC$ whose generic fiber is $Q$.  
In Section~\ref{sec:criterion}, we verify the conditions of \cite{CTO} for these cubics, so we obtain a nonzero unramified
cohomology class on $Y_0$. More generally, Proposition~\ref{prop:criterion} gives a family of such examples.
In Section~\ref{sec:specialization}, we apply a specialization theorem of Schreieder \cite{Sch} to deduce Theorem~\ref{thm:main}.

Using similar arguments, we also obtain stable irrationality for very general quartic sixfolds
singular along a $3$-plane and for very general complete intersections
of types $(2,2,3)$ in $\PP^9_\CC$ and $(2,2,2,2)$ in $\PP^{10}_\CC$
(Propositions~\ref{prop:doubleplane}, \ref{prop:223}, and~\ref{prop:fourquadrics}).

We work over $\CC$ for simplicity, but the arguments apply over any algebraically closed field of characteristic zero.


\section{Preliminaries}\label{sec:cohomology}
In this section, we recall a few cohomological facts which will be needed later on. We refer to \cite{CTO}, \cite{CTS}, and \cite{Sch} for more background.

We work with fields $F\supset\CC$ and write $H^r(F,\ZZ/2)$ for the $r$-th Galois cohomology group (in characteristic zero, we canonically identify
$\mu_2^{\otimes r}$ with $\ZZ/2$).  We identify $H^1(F,\ZZ/2)$ with $F^\times/(F^\times)^2$, using the Kummer sequence. 
For $a_i\in F^\times$, we write
$$
 (a_1,\ldots,a_r)=(a_1)\cup \cdots \cup (a_r)\in H^r(F,\ZZ/2).
$$
These symbols are multilinear and invariant under permutations, and are zero when one of the entries is a square or when two entries are equal. They also satisfy the Steinberg relation
$(a,1-a)=0$ for $a\neq 0,1$. We will also identify $H^2(F,\ZZ/2)$ with the 2-torsion part of the Brauer group,
$\Br(F)[2]$. Then $(a,b)$ is the class of the quaternion algebra associated to the conic $U^2-aV^2-bW^2=0$ in $\PP^2_F$. In particular, $(a,b)=0$
if and only if this conic has an $F$-point.

Let $\nu\colon F^\times \to \ZZ$ be a discrete valuation over $\CC$ (meaning, $\nu(c)=0$ for every $c\in\CC^\times$).
We write $A_\nu$ for its valuation ring and $\kappa_\nu$ for its residue field.
For each $r\geq 1$, there is a residue map
$$
\dd_\nu\colon H^r(F,\ZZ/2)
\to H^{r-1}(\kappa_\nu,\ZZ/2).
$$
If the $u_i$ are units and $\pi$ is a uniformizer, then
\begin{equation}\label{eq:resrules}
\dd_\nu(u_1,\ldots,u_r)=0,\qquad \dd_\nu(u_1,\ldots,u_{r-1},\pi) =(\bar u_1,\ldots,\bar u_{r-1}),
\end{equation}
where $\bar u_i$ is the class of $u_i$ in $\kappa_\nu$. A class $\sigma\in H^r(F,\ZZ/2)$ is called
{\em unramified} if $\dd_\nu(\sigma)=0$ for every discrete
valuation $\nu$ on $F$ over $\CC$. We denote the subgroup of these by $H^r_{\nr}(F,\ZZ/2)$ (we omit the base field $\CC$ from the notation, for simplicity).


We write $\langle d_1,\ldots,d_r\rangle$ for the quadratic form $\sum d_iX_i^2$. For convenience, we include the following well-known lemma.

\begin{lemma}\label{lem:norms}
Let $a,b,d\in F^\times$.
\begin{enumerate}[(i)]
\item \label{adaifnasd} If $d=x^2-ay^2$, then $(a,d)=0$.
\item \label{adaifnasd2}If $d=x^2-ay^2-bz^2+abt^2$, then $(a,b,d)=0$.
\end{enumerate}
The statement \ref{adaifnasd2} also holds for nonzero values of
$\langle1,a,b,ab\rangle$.
\end{lemma}
\begin{proof}
\ref{adaifnasd}: If $x,y\ne0$, let $e=a(y/x)^2$. Then $d=x^2(1-e)$, so working modulo squares, $(a,d)=(e,1-e)=0$. If $y=0$, then $d$ is a square. If $x=0$, then $(a,d)=(a,-ay^2)=(a,-a)=0$.

\ref{adaifnasd2}: Let $R=x^2-ay^2$ and $S=z^2-at^2$, so that $d=R-bS$. If $R,S\ne0$, then $(a,R)=(a,S)=0$, by \ref{adaifnasd}. 
Then with $e=S/R$, we have
$$
 (a,b,d)=(a,b,1-be)=(a,be,1-be)+(a,e,1-be)=0.
$$
Here the first term vanishes by the Steinberg relation and the second because $(a,e)=(a,S)-(a,R)=0$.
If $S=0$, then $d=R$ is nonzero, so $(a,d)=0$ as above, and hence $(a,b,d)=0$.
If $R=0$, then $d=-bS\ne0$, and we get $(a,b,d)=(a,b,b)+(a,b,S)=0$. The last statement follows because $-1$ is a square in $F$.
\end{proof}


In the arguments below, we will usually work in the henselization $A_\nu^h$ of the valuation ring. Passing to the henselization preserves the residue field and is compatible with residue maps (see \cite[p.~143]{CTO}). We will use the following two consequences of Hensel's lemma. Let $R$ be a henselian discrete valuation ring with residue field $\kappa$ of characteristic different from 2. If $x\in R^\times$ and $\bar{x}\in\kappa$ is a square, then $x$ is a square in $R$.
Also, for $u,v\in R^\times$, the equality $(\bar u,\bar v)=0$ over $\kappa$ implies $(u,v)=0$ over the fraction field of $R$. Indeed, a point on the smooth conic $X^2-\bar uY^2-\bar vZ^2=0$ over $\kappa$ lifts to an $R$-point.

\section{Quadrics of Colliot-Th\'el\`ene--Ojanguren type}\label{sec:criterion}
We now work over the field $K=\CC(x,y,z)$.  We will use the following criterion to construct the unramified cohomology class needed for Theorem \ref{thm:main}. 

\begin{proposition}[Colliot-Th\'el\`ene--Ojanguren]\label{prop:unramified}
For $g_1,g_2\in K^\times$, let $Q$ be the quadric
\eqref{eq:Q} and let $\alpha_i=(x,y,g_i)$. Suppose that
\begin{enumerate}[(i)]
\item \label{CTO1} $\alpha_1,\alpha_2\ne0$.
\item \label{CTO2} $\dd_\nu\alpha_1=0$ or $\dd_\nu\alpha_2=0$ for every discrete valuation $\nu$ on $K$ over $\CC$.
\end{enumerate}
Then $\alpha_1$ and $\alpha_2$ pull back to the same nonzero
unramified class $\gamma\in H^3(K(Q),\ZZ/2)$.
\end{proposition}

\begin{proof}This is \cite[Proposition~3.1]{CTO}. For completeness, we sketch a proof, following \cite[Proposition~17]{Sch}. Recall that a {Pfister form} is a tensor product of binary forms $\langle1,-a_i\rangle$ with $a_i\in K^\times$. A {Pfister neighbour} is, up to scaling and an invertible linear change of variables, a subform $q_1$ of a Pfister form $q_2$ with $\rank(q_1)>\tfrac12 \rank(q_2)$.
In our case, the rank-5 form of $Q$ is a Pfister neighbour of the rank-8 Pfister form
$
\langle1,-x\rangle\otimes\langle1,-y\rangle
\otimes\langle1,-g_1g_2\rangle.$ For such a neighbour, Arason's theorem \cite[Satz~5.6]{Ar} shows that the kernel of
$H^3(K,\ZZ/2)\to H^3(K(Q),\ZZ/2)$ is generated by the symbol $(x,y,g_1g_2)=\alpha_1+\alpha_2$. The classes $\alpha_1$ and $\alpha_2$ therefore pull back to the same class $\gamma$ in $H^3(K(Q),\ZZ/2)$.  The class $\gamma$ is nonzero by condition~\ref{CTO1}.  
To see that it is unramified, let $\nu$ be a discrete valuation of $K(Q)$ over $\CC$.
If the restriction of $\nu$ to $K$ is nontrivial, condition~\ref{CTO2} and compatibility of residues give $\dd_\nu(\gamma)=0$.
And if the restriction is trivial, all the entries of the symbols are units in $A_\nu$, so the residue is zero by \eqref{eq:resrules}.
\end{proof}


The two classes $(x,y,g_1)$ and $(x,y,g_2)$ are ramified in $H^3(K,\ZZ/2)$. As the proof shows, however, their
pullbacks to the quadric coincide, and the assumptions on the residues give that this class is unramified.

The next result gives a family of pairs $g_1,g_2$ satisfying this criterion.
\begin{proposition}\label{prop:criterion}
Let $a,b,A,B\in\CC[z]$ and $c_1,c_2\in\CC^\times$, and define
\begin{equation}\label{eq:shape}
\begin{aligned}
g_1&=1+ax+B^2y+c_1xy,\\
g_2&=1+A^2x+by+c_2xy.
\end{aligned}
\end{equation}
Assume that $a,b$ are not squares and that at least one of them has even degree. Then the classes $\alpha_1=(x,y,g_1)$ and $\alpha_2=(x,y,g_2)$ satisfy conditions~\ref{CTO1} and~\ref{CTO2} of
Proposition~\ref{prop:unramified}.
\end{proposition}
First of all, the assumptions imply that $g_1,g_2$ are irreducible. Indeed, a polynomial of the form
$1+px+qy+cxy$ cannot factor over $\CC(z)$ unless $pq=c$, and this does not happen here because $c_1,c_2$ are constants.
$g_1$ and $g_2$ are moreover different because $a\ne A^2$, so they are coprime.

The equations and the proof below are inspired by Hassett--Pirutka--Tschinkel \cite{HPT}
and Colliot-Th\'el\`ene's exposition \cite{CT}.  In that example, the discriminant consists of three lines and a conic tangent to each. The tangency makes
the equation of the conic a square on each line, which is used to prove that the cohomology class is unramified.

For the polynomials above, we have the following restrictions to the coordinate planes:
$$
\begin{array}{c|cc}
 & x=0 & y=0\\ \hline
g_1 & 1+B^2y & 1+ax\\
g_2 & 1+by & 1+A^2x
\end{array}.
$$Here the residue of $\alpha_1=(x,y,g_1)$ at the generic point of $\{x=0\}$ is equal to $(y,1+B^2y)$ in $\Br(\CC(y,z))[2]$, and this is zero by Lemma~\ref{lem:norms}\ref{adaifnasd}. Likewise,
$\alpha_2=(x,y,g_2)$ has zero residue along $\{y=0\}$. 
The two other restrictions will be used to show that $\alpha_1,\alpha_2$ are nonzero. 
\begin{proof}
We first check \ref{CTO1}. Along $\{y=0\}$, the residue $\partial_y\alpha_1$ is
$(x,1+ax)$ in $\Br(\CC(x,z))[2]$. In the $xz$-plane, the curve $C=\{1+ax=0\}$ is integral
with function field $K(C)\simeq \CC(z)$ where $x=-1/a$. The residue of $(x,1+ax)$ at $C$ is then given by the class of $-1/a$ in $\CC(z)^\times/(\CC(z)^\times)^2$, which is nonzero because $a$ is not a square. This shows that $\alpha_1\ne0$. The same argument with $\alpha_2$ along $\{x=0\}$ shows that $\alpha_2\ne0$.

For condition \ref{CTO2}, let $\nu$ be a discrete valuation of $K$ over $\CC$ and let
$R=A_\nu^h$ be the henselization of $A_\nu$, with fraction field $K_\nu^h$. There are three cases:

\smallskip\noindent\textit{Case 1: $x,y,z\in R$.}
Suppose first that $\nu(x)>0$. If also $\nu(y)>0$, then $\bar g_1=1$ in $\kappa_\nu$,
so Hensel's lemma shows that $g_1$ is a square in $K_\nu^h$, and $\alpha_1=(x,y,g_1)=0$. 
Otherwise, $y$ is a unit and $\bar g_1=1+\bar y\bar B^2$.
If $g_1$ is a unit, the smooth conic $U^2-yV^2-g_1W^2=0$ over $R$ has the $\kappa_\nu$-point $(1:i\bar B:1)$.
This lifts to $R$ by Hensel's lemma, so $(y,g_1)=0$ over
$K_\nu^h$. Hence $\alpha_1=(x)\cup(y,g_1)=0$.
If $g_1$ is not a unit, then $\bar{g_1}=0$ in $\kappa_\nu$, so $\bar y=-1/\bar B^2$ is a square. By Hensel's lemma,
$y$ is a square in $K_\nu^h$ and again $\alpha_1=0$. A similar argument works if $\nu(y)>0$, in which case $\alpha_2=0$.

Finally, suppose that $x,y$ are both units. If some $g_i$ is a unit, all
entries of $\alpha_i$ are units and $\dd_\nu\alpha_i=0$ by \eqref{eq:resrules}.
Otherwise the center $Z\subset\AA^3_\CC$ of $\nu$ is contained in $\{g_1=g_2=0\}$.
As $g_1,g_2$ are coprime, we must have $\dim Z\le1$. Here $\Br(\CC(Z))[2]=0$ by Tsen's
theorem if $Z$ is a curve, and because $\CC$ is algebraically closed when $Z$ is a point. The class $(\bar x,\bar y)$ is the pullback of a class
in $\Br(\CC(Z))[2]$, so it is zero. Hensel's lemma therefore gives $(x,y)=0$ over $K_\nu^h$,
so $\alpha_1$ and  $\alpha_2$ both vanish there.

\smallskip\noindent\textit{Case 2: $\nu(z)<0$.} If $\deg a=2m$, then $a/z^{2m}$ is a unit whose
reduction to $\kappa_\nu$ is the leading coefficient of $a$. This coefficient is a complex number, so it is a square. Hensel's lemma
shows that $a$ is a square in $K_\nu^h$. Writing $a=p^2$ and $c_1=q^2$, we obtain $g_1=1+xp^2+yB^2+xyq^2.$ 
Lemma~\ref{lem:norms} then gives $\alpha_1=0$. If $\deg a$ is odd, then $\deg b$ is even, and the same argument gives $\alpha_2=0$.

\smallskip\noindent\textit{Case 3: $\nu(z)\ge 0$ and either $\nu(x)< 0$ or $\nu(y)<0$.} 
Then $z\in R$.
If $\nu(x)<0$, then $c_2+b/x$ is a unit with reduction $c_2$, so it is a square by Hensel's lemma. Writing it as
$q^2$ for $q\in R^\times$, we obtain
$
g_2=1+xA^2+xyq^2.
$
Lemma~\ref{lem:norms} therefore gives $\alpha_2=0$. Similarly, if $\nu(y)<0$, the unit $c_1+a/y$ is a square and the same argument gives
$\alpha_1=0$.
\end{proof}

\section{Quartic sixfolds}\label{sec:quartic}
In this section, we construct a quartic sixfold with function field $K(Q)$, where $Q$ is the quadric \eqref{eq:Q}.
Let $a,b,A,B\in\CC[z]$ and $c_1,c_2\in\CC^\times$
satisfy the conditions of Proposition~\ref{prop:criterion}. We also make the assumption that
$A,B,a-A^2,b-B^2$ have degree at most 1.

With these assumptions, $g_1$ and $g_2$ have the same degree 3 part, so $g_1-g_2$ has degree at most 2. Let $h_1=(g_1+g_2)/2$ and $h_2=(g_1-g_2)/2$. 
Since $g_1g_2=h_1^2-h_2^2$, the substitution $s=w+h_1$ gives the affine quartic
\begin{equation}\label{eq:affine}
w^2+2h_1w-xt^2-yu^2+xyv^2+h_2^2=0.
\end{equation}
This quartic is irreducible, because $Q$ is irreducible.
 \begin{example}
A sparse example is obtained by taking $a=2z$, $b=z^2-2z$, $A=0$, $B=z$, and $c_1=c_2=1$.
This leads to the affine quartic \eqref{eq:introaffine} in the introduction. Among all the $a,b,A,B,c_1,c_2$ satisfying the above criterion, this tuple appears to give the fewest monomials (twelve) in the corresponding quartic sixfold.
\end{example}

\begin{example}\label{ainbiaedrnfbnsdcn}
The tuple $(a,b,A,B)=(2-z^2,2+z^2,iz,z)$ and $c_1=c_2=1$ 
leads to a quartic which has only three monomials of degree 4.
\begin{equation}\label{asidvnairndasdfwe}
w^2+2(1+x+y+xy-xz^2+yz^2)w
-xt^2-yu^2+xyv^2+(x-y)^2=0.
\end{equation}We will use this quartic in Section \ref{sec:ci}.
\end{example}

Choose homogeneous coordinates $x_0,\ldots,x_7$ on $\PP^7_\CC$ so that
$x,y,z,w,t,u,v$ correspond to $x_1/x_0,\ldots,x_7/x_0$ respectively. The homogenization of \eqref{eq:affine} is
\begin{equation}\label{eq:F0}
 F_0=x_0^2x_4^2+2H_1 x_4-x_1x_0x_5^2-x_2x_0x_6^2+x_1x_2x_7^2+H_2^2,
\end{equation}
where $H_1,H_2\in \CC[x_0,x_1,x_2,x_3]$ are the homogenizations of
$h_1,h_2$ to degrees 3 and 2, respectively.
 
The polynomial $F_0$ is irreducible, so the hypersurface $X_0=\{F_0=0\}$ is integral and is the projective closure of \eqref{eq:affine} in $\PP^7_\CC$. Looking at the monomials appearing in $F_0$, we see that $X_0$ has multiplicity 2 along the 3-plane $L=\{x_0=x_1=x_2=x_3=0\}$. 

Write $B=\PP^3_\CC$ with homogeneous coordinates $u_0,u_1,u_2,u_3$.
Blowing up $L$ resolves the projection map $\PP^7_\CC\dashrightarrow B$  and we obtain morphisms $p\colon M=\Bl_L\PP^7_\CC\to\PP^7_\CC$ and  $q\colon M\to B$.
The morphism $q$ is a $\PP^4$-bundle, given by
$M=\PP(\O_B^{\oplus4}\oplus\O_B(1))$ (where we use the quotient convention for projective bundles). 
The hyperplane bundle is given by $\O_M(1)=p^*\O(1)$.

We choose global sections of line bundles on $M$ as follows.
The first four summands of $\O_B^{\oplus4}\oplus\O_B(1)$
define sections $v_0,v_1,v_2,v_3$ of $\O_M(1)$. The last summand defines a section
$e$ of $\O_M(1)\otimes q^*\O_B(-1)$ whose zero scheme is the exceptional divisor.
We also write $u_i$ for the pullbacks of the coordinate sections on $B$.
Then $\O_M(1)$ is globally generated by $u_0e,u_1e,u_2e,u_3e,v_0,v_1,v_2,v_3$, and these
define the blow-up morphism $p\colon M\to \PP^7_\CC$.
Since $X_0$ is integral, the strict transform $Y_0$ is the blow-up of $X_0$ along $L$.
Substituting into \eqref{eq:F0}, we see that $Y_0\subset M$
is defined by the equation
\begin{equation}\label{eq:q0}
u_0^2v_0^2-u_1u_0v_1^2-u_2u_0v_2^2+u_1u_2v_3^2
+2H_1(u_0,u_1,u_2,u_3)v_0e+H_2(u_0,u_1,u_2,u_3)^2e^2=0.
\end{equation}
Identifying $K=\CC(x,y,z)=\CC(B)$, the substitutions above  show that the generic fiber of $Y_0\to B$ is given by $Q$, so $\CC(Y_0)=K(Q)$.

%


Let $G_i=u_0^3g_i(u_1/u_0,u_2/u_0,u_3/u_0)$ be the 
homogenization of $g_i$. The discriminant of $Y_0\to B$ is given by $-u_0^2u_1^2u_2^2G_1G_2$.
For the cubics in \eqref{eq:sparse}, its support consists of three planes and two cubic surfaces. 
One can compare this with the examples of
Artin--Mumford \cite{AM} (two plane curves of degree three),
Colliot-Th\'el\`ene--Ojanguren \cite{CTO} (a union of planes), and
Hassett--Pirutka--Tschinkel \cite{HPT} (a conic and three tangent lines).


\begin{remark}
The quartic $X_0$ is also singular along a second $3$-plane, namely $\{x_0=x_1=x_2=x_4=0\}$.
Projecting from this 3-plane, we obtain a different quadric threefold bundle $Y_0'\to \PP^3_\CC$.
The discriminant is again highly reducible. In the special case $c_1=c_2$, it 
factors as $x_0^3x_1^2x_2^2x_4 R$ where $R$ is a quartic.
\end{remark}

\section{Proof of Theorem \ref{thm:main}}\label{sec:specialization}
To prove the main theorem, we use the following specialization result for the decomposition of the diagonal property.

\begin{proposition}[Schreieder]\label{prop:specialization}
Let $Z$ be an integral complex projective variety and let $\tau\colon\widetilde Z\to Z$ be a projective resolution
which is an isomorphism over an open subset $U\subset Z$. Suppose that
$D=\widetilde Z\setminus\tau^{-1}(U)$ is a simple normal
crossings divisor. If, for some $r\ge 1$, there is a nonzero class
$
 \beta\in H^r_{\nr}(\CC(Z),\ZZ/2)$ which restricts to zero in $H^r(\CC(D_i),\ZZ/2)$ for every
irreducible component $D_i$ of $D$, then:
\begin{enumerate}[(i)]
\item\label{specialization:special}
$Z$ admits no decomposition of the diagonal.
\item\label{specialization:general}
For any flat projective family with geometrically integral fibers
over an integral complex base that contains $Z$ as a fiber,
a very general member admits no decomposition of the diagonal.
\end{enumerate}
\end{proposition}
\begin{proof}
Part~\ref{specialization:special} follows from \cite[Proposition~26]{Sch} and its proof. Part~\ref{specialization:general} follows from the specialization statement in that proposition and \cite[Lemma~8]{Sch}.
\end{proof}
Here, for a prime divisor $D$ on a smooth variety, $\beta|_D$ is obtained by extending the unramified class $\beta$ over the local ring at the generic point of $D$ and then restricting to its residue field $\CC(D)$.

The advantage of the above proposition is that it allows us to work with a resolution of $Y_0$ without computing it explicitly. This is useful here because even after blowing up $L$, the sixfold $Y_0$ has a rather complicated singular locus. For the cubics $g_1,g_2$ in \eqref{eq:sparse}, a computation in \verb|Macaulay2| shows that it has nine components: 3 threefolds, 4 surfaces, and 2 curves. None of these dominate the base $B$.

By \cite[Proposition~28]{Sch}, the class $\gamma$ restricts to zero on every divisor of
a smooth projective model that does not dominate $B$. This takes care of the exceptional divisors which arise from 
desingularizing $Y_0$. Let $\pi=p|_{Y_0}\colon Y_0\to X_0$ and let
$E=\pi^{-1}(L)$ be the exceptional divisor. Since $E$ dominates $B$, we need to check the restriction
to this divisor directly.
\begin{lemma}\label{aslviunasriindn}
$\gamma|_E=0$ in $H^3(\CC(E),\ZZ/2)$.
\end{lemma}
\begin{proof}
The exceptional divisor $E$ is obtained by setting $e=0$ in \eqref{eq:q0}, so
$$
E=\{u_0^2v_0^2-u_1u_0v_1^2-u_2u_0v_2^2+u_1u_2v_3^2=e=0\}.
$$This divisor is integral and its generic fiber over $B=\PP^3_\CC$
is the smooth quadric surface
\begin{equation}\label{eq:exceptional}
E_K=\{v_0^2-xv_1^2-yv_2^2+xyv_3^2=0\}\subset\PP^3_K.
\end{equation}
On the chart $v_3=1$,  the quadric has affine equation $v_0^2-xv_1^2-yv_2^2+xy=0$. We work in the function field $\CC(E)$. Let $c=v_2^2-x$, so that $v_0^2-xv_1^2=yc$. Both $c=v_2^2-x$ and
$yc=v_0^2-xv_1^2$ are nonzero and of the form $\lambda^2-x\mu^2$, so $(x,c)=(x,yc)=0$ by Lemma~\ref{lem:norms}. 
This gives $(x,y)=(x,c\cdot yc)=(x,c)+(x,yc)=0$ in $H^2(\CC(E),\ZZ/2)$.  
Now, since $E$ dominates $B$, the elements $x,y,g_1\in K^\times$ 
are units in the local ring of $Y_0$ at the generic point of $E$.
It follows that $\gamma|_E=(x,y,g_1)=(x,y)\cup (g_1)=0$ in $H^3(\CC(E),\ZZ/2)$.
\end{proof}
\begin{proof}[Proof of Theorem \ref{thm:main}]
We use Proposition~\ref{prop:specialization} with $Z=X_0$ and $\beta=\gamma$. The class $\gamma$ is nonzero and unramified by
Propositions~\ref{prop:criterion} and~\ref{prop:unramified}.

Let $B^\circ\subset B$ be the complement of the discriminant and let $Y_0^\circ=(q|_{Y_0})^{-1}(B^\circ)$.
Both $Y_0^\circ$ and $E\cap Y_0^\circ$ are smooth, since their fibers over $B^\circ$ are smooth quadrics. We may therefore choose a projective log resolution
$\rho\colon \widetilde Y_0\to Y_0$ of the pair $(Y_0,E\cup(Y_0\setminus Y_0^\circ))$ which is an isomorphism over $Y_0^\circ$. In Proposition~\ref{prop:specialization}, we take
$\widetilde Z=\widetilde Y_0$ and
$\tau=\pi\circ\rho:\widetilde Y_0\to X_0$.
This is a resolution of $X_0$ and is an isomorphism over the open set $
U=X_0\setminus\bigl(L\cup\pi(Y_0\setminus Y_0^\circ)\bigr)$ and the complement $\widetilde Y_0\setminus\tau^{-1}(U)$ is a simple
normal crossings divisor. The only component of this divisor which dominates $B$
is the strict transform $\widetilde E$ of $E$, as $\rho$ is an isomorphism over $Y_0^\circ$.
Since $\CC(\widetilde E)=\CC(E)$, Lemma \ref{aslviunasriindn} shows that $\gamma|_{\widetilde E}=0$.
On every other boundary component, $\gamma$ restricts to zero
by \cite[Proposition~28]{Sch}.  

Proposition~\ref{prop:specialization}~\ref{specialization:special} applied to $Z=X_0$ now shows that
$X_0$ admits no decomposition of the diagonal. Finally, applying Proposition~\ref{prop:specialization}~\ref{specialization:general}  to the universal family
of integral quartics in $\PP^7_\CC$, with special fiber $X_0$ and resolution $\widetilde Y_0\to X_0$, we obtain Theorem~\ref{thm:main}. 
\end{proof}


%

In fact, the same argument also proves the following.

\begin{proposition}\label{prop:doubleplane}
Let $L\subset\PP^7_\CC$ be a $3$-plane. A very general quartic
hypersurface with multiplicity 2 along $L$ admits no
decomposition of the diagonal and is not stably rational.
\end{proposition}

\begin{proof}
This follows as above, by applying Proposition~\ref{prop:specialization}\,\ref{specialization:general}
to the universal family of integral quartics double along $L$,
with special fiber $X_0$.
\end{proof}

\begin{remark}
If one only wants stable irrationality, the specialization step above 
can be shortened using \cite[Theorem~3.14]{NO}.
Let $M=\Bl_L\PP^7_\CC$, let $H$ be the pullback of the hyperplane
class, and let $E$ be the exceptional divisor. The linear system $|4H-2E|$ is basepoint-free, contains $Y_0$,
and has a smooth general member. The open set $Y_0^\circ$ is smooth and maps to $B$ with generic
fiber the proper quadric $Q$. Since $g_1,g_2$ satisfy the conditions of
Proposition~\ref{prop:unramified}, \cite[Theorem~39]{Sch} directly shows that a very general
member admits no decomposition of the diagonal.
To obtain stable irrationality for very general quartic sixfolds, subdivide the Newton polytope of a general quartic by $a_0+\cdots+a_3=2$, where $a_i$ is the exponent of $x_i$.
The two maximal cells are the polytopes of $4H-2E$ on the blow-ups of $\PP^7_\CC$ along
$L$ and $\{x_4=x_5=x_6=x_7=0\}$, respectively. For very general coefficients, the corresponding hypersurfaces
$Y_1,Y_2$ are therefore stably irrational. The common face defines a divisor $Y_{12}$ of bidegree $(2,2)$
in $\PP^3_\CC\times\PP^3_\CC$. By \cite[Theorem~3.14]{NO}, the stable birational volume of the resulting degeneration is
$[Y_1]+[Y_2]-[Y_{12}].$ Since neither $Y_i$ is stably rational, the coefficient of $[\Spec\CC]$ in this expression is either $0$ or $-1$. 
The volume therefore differs from $[\Spec\CC]$ in the free abelian group on stable birational types, so the very general quartic sixfold is not stably rational.
\end{remark}

\begin{remark}
As explained in \cite{CTO}, these types of sixfolds are unirational. Indeed, over the degree 2 extension
$K(\xi)=\CC(\xi,y,z)$, where $\xi^2=x$, the substitutions
$T'=\xi T$, $V'=\xi V$ give the equation
$S^2-{T'}^2-yU^2+y{V'}^2-g_1g_2W^2=0$.
This smooth quadric has the rational point
$(S:T':U:V':W)=(1:1:0:0:0)$, so it is rational over $\CC(\xi,y,z)$.
This means that $X_0$ admits a unirational parametrization of degree 2.
We do not know whether a general quartic hypersurface in $\PP^7_\CC$ singular along a $3$-plane is unirational.
\end{remark}

\section{Complete intersections}\label{sec:ci}
The same construction gives irrationality results for other complete intersections.
\begin{proposition}\label{prop:223}
A very general complete intersection of type $(2,2,3)$ in $\PP^9_\CC$ admits no decomposition of the diagonal.
\end{proposition}

\begin{proof}
The basic idea is to introduce two extra variables and rewrite the quartic as a cubic.
We use the quartic \eqref{asidvnairndasdfwe} of Example \ref{ainbiaedrnfbnsdcn}, with degree 4 part 
$xyv^2+2(y-x)z^2w$. Letting $r=xy$ and $s=(y-x)w$, the quartic can be written
$$
c=w^2+2(1+x+y+xy)w-xt^2-yu^2+rv^2+2sz^2+(x-y)^2=0.
$$
Let $x_0,\ldots,x_9$ be homogeneous coordinates on $\PP^9_\CC$
such that $x_1/x_0,\ldots,x_9/x_0$ correspond to $x,y,z,w,t,u,v,r,s$.
Let $C$ be the homogenization of $c$, and define
$$
Z=\{x_0x_8-x_1x_2=x_0x_9-(x_2-x_1)x_4=C=0\}.
$$
A calculation in \verb|Macaulay2| checks that $Z$ is an integral complete intersection in $\PP^9_\CC$.
On $x_0\ne0$, we may solve for $x_8,x_9$ and then $C=0$ reduces to \eqref{asidvnairndasdfwe}.

Let $\Gamma\subset Z\times B$ be the closure of the graph of the projection $Z\dashrightarrow B=\PP^3_\CC$ onto $(x_0:x_1:x_2:x_3)$. Let $h_1=(g_1+g_2)/2$. 

Over $K=\CC(x,y,z)$, the map sending
$(S:T:U:V:W)$ to $(W:xW:yW:zW:S-h_1W:T:U:V:xyW:(y-x)(S-h_1W))$ is a linear map $\PP^4_K\to \PP^9_K$, which identifies
$Q$ with $\Gamma_K$. The inverse image $D$ of $Z\cap\{x_0=0\}$ has generic
fiber $E_K=Q\cap\{W=0\}$.

Choose an open set $B^\circ\subset B$ over which
$\Gamma\to B$ and $D\to B$ are smooth. Let $U\subset Z$
be the open subset where $x_0\ne0$ and the projection to $B$ lies in $B^\circ$.
As in the proof of Theorem~\ref{thm:main}, we may blow up $\Gamma$ 
without changing $\Gamma|_{B^\circ}$ to obtain a resolution
$\tau\colon\widetilde Z\to Z$ that is an isomorphism over $U$ and has simple normal crossings boundary
$\widetilde Z\setminus\tau^{-1}(U)$. The only component dominating $B$ has generic fiber $E_K$.
The nonzero unramified class $\gamma$ 
restricts to zero on this component by
Lemma~\ref{aslviunasriindn}, and on every other boundary
component by \cite[Proposition~28]{Sch}.
Applying Proposition~\ref{prop:specialization} to the universal family of integral complete intersections of
type $(2,2,3)$, we obtain the result.
\end{proof}

We next consider complete intersections of quadrics. The rationality problem for complete intersections has been widely studied.
Beauville \cite{Bea} proved that a smooth intersection of three quadrics in $\PP^6_\CC$ is not rational, and Hassett--Tschinkel \cite{HT} proved that a very
general such threefold is not stably rational. Hassett--Pirutka--Tschinkel \cite{HPT3}
proved that a very general intersection of three quadrics
in $\PP^7_\CC$ does not admit a decomposition of the diagonal. 
See \cite{NO} for results in higher dimensions.

The first open case concerns complete intersections of four quadrics
in $\PP^{10}_\CC$. Before treating this case, we first study a different family of quadric threefold bundles over $\PP^3_\CC$,
complete intersections of divisors of bidegrees $(1,1),(1,1),(1,1),(1,2)$ in
$\PP^3_\CC\times\PP^7_\CC$. This construction follows \cite[Proposition~6 and Section~3.1]{HPT3} closely.

\begin{proposition}\label{prop:incidenceCI}
A very general complete intersection of three $(1,1)$-divisors and one $(1,2)$-divisor 
in $\PP^3_\CC\times\PP^7_\CC$ admits no decomposition of the diagonal.
\end{proposition}

\begin{proof}
We use the same quartic \eqref{asidvnairndasdfwe} as before. To find a model as a complete intersection in $\PP^3_\CC\times \PP^7_\CC$, we add three extra variables and rewrite the equation so that all coefficients have degree at most $1$
in $x,y,z$. 

Replacing $w,t,u,v$ by 
$(x-y)v_1,\ldots, (x-y)v_4$ and dividing by $(x-y)^2$, we get
$$
v_1^2+1-xv_2^2-yv_3^2+xyv_4^2
+2\bigl(1+y-z^2+\frac{(1+y)^2}{x-y}\bigr)v_1=0.
$$
Defining $v_5=xv_4$ and $v_6=z$, we may rewrite $xyv_4^2$ as $yv_4v_5$ and $z^2v_1$ as $zv_1v_6$ respectively.
The fraction can be taken care of by defining $v_7=(1+y)/(x-y)$, so that it equals 
$(1+y)v_7$. With these definitions, the equation becomes
$$
v_1^2+1-xv_2^2-yv_3^2+yv_4v_5
+2v_1\bigl((1+y)(1+v_7)-zv_6\bigr)=0.
$$
To homogenize this to $\PP^3_\CC\times \PP^7_\CC$,  use $u_0,u_1,u_2,u_3$ as homogeneous coordinates on 
$\PP^3_\CC$, with $x=u_1/u_0$, $y=u_2/u_0$, $z=u_3/u_0$, and
$v_0,\ldots,v_7$ on $\PP^7_\CC$ so that the above equations take place in the chart
$u_0=v_0=1$. This gives a subscheme $Z\subset\PP^3_\CC\times\PP^7_\CC$
defined by the three equations of bidegree $(1,1)$,
$$
u_0v_5=u_1v_4,\quad
u_0v_6=u_3v_0,\quad
(u_1-u_2)v_7=(u_0+u_2)v_0,
$$
and one of bidegree $(1,2)$,
$$
u_0(v_1^2+v_0^2)-u_1v_2^2-u_2v_3^2+u_2v_4v_5
+2v_1\bigl((u_0+u_2)(v_0+v_7)-u_3v_6\bigr)=0.
$$
A calculation in \verb|Macaulay2| checks that $Z$ is an integral
complete intersection.

Over $K$, we may solve the three linear equations for $v_5,v_6,v_7$, so $Z$ is birational to the quartic \eqref{asidvnairndasdfwe}.
Moreover, doing the above coordinate changes in reverse, we see that the generic fiber of $Z\to \PP^3_\CC$
is the quadric $Q$. The specialization argument now works as before, applying \cite[Theorem~39]{Sch} to the universal family of integral complete intersections of these bidegrees.
\end{proof}


%
%
%

\begin{proposition}\label{prop:fourquadrics}
A very general complete intersection of four quadrics in $\PP^{10}_\CC$
admits no decomposition of the diagonal.\end{proposition}

\begin{proof}
A general complete intersection of four quadrics
in $\PP^{10}_\CC$ contains a $2$-plane (see e.g., \cite[Theorem~2.1]{DM}). For such a plane $\Lambda\subset \PP^{10}_\CC$, consider a very general complete intersection $X$ containing $\Lambda$.
For a very general such pair, the geometric construction of \cite[Proposition~6]{HPT3} shows that $X$ is birational to a very
general complete intersection of bidegrees $(1,1),(1,1),(1,1),(1,2)$ in $\PP^3_\CC\times\PP^7_\CC$. Hence the result follows from Proposition~\ref{prop:incidenceCI}.
\end{proof}
Combining this with the methods of the paper \cite{NO}, we also improve the irrationality bounds for intersections of quadrics in higher dimension. 
\begin{corollary}\label{cor:manyquadrics}
Let $n\ge3$ and $r\ge4$ be integers with $r\ge n-2$.
A very general complete intersection of $r$ quadrics in
$\PP^{n+r}_\CC$ is not stably rational.
\end{corollary}
\begin{proof}
The cases $n\le r+1$ follow from \cite[Corollary~7.8]{NO}.
For $n=r+2$, Proposition~\ref{prop:fourquadrics} gives
the case $r=4$. Applying \cite[Theorem~7.7]{NO} with
$d=2$ adds one quadric and increases the dimension by one,
so the result follows by induction. \end{proof}

\begin{remark}Another easy corollary of Proposition \ref{prop:incidenceCI} is that a very general bidegree $(4,5)$-divisor in
$\PP^3_\CC\times\PP^7_\CC$ is not stably rational. This also follows by degenerating the divisor into a union of three $(1,1)$-divisors and one $(1,2)$-divisor, and arguing as in \cite{NO}.\end{remark}

A \verb|Macaulay2| file verifying some of the calculations in this paper is available at
\begin{center}
\url{https://www.mn.uio.no/math/personer/vit/johnco/papers/quartic6.m2}
\end{center}

%
%


\subsection*{AI disclosure}We used ChatGPT-5.6/6 to search for equations, write \verb|Macaulay2| code and check arguments. It in particular contributed with a calculation that showed that for certain explicit choices of $g_1,g_2$, the corresponding degree class 3 is unramified. This led to Proposition~\ref{prop:criterion}, whose proof simplifies and generalizes that calculation. We thank OpenAI for access through its Researcher Access Program.

\subsection*{Acknowledgements}
We thank Jean-Louis Colliot-Th\'el\`ene and Stefan Schreieder for very helpful comments.

\end{document}